\documentclass[11pt]{amsart}

\usepackage{hyperref} 
\usepackage[dvips]{graphicx}
\usepackage{color}
\usepackage{multicol}
\usepackage[T1]{fontenc}
\usepackage{mathptmx}

\newcommand{\bbar}{\begin{pmatrix}}
\newcommand{\ebar}{\end{pmatrix}}

\newcommand{\uu}{{\Lambda G_\sigma}}
\newcommand{\bdm}{\begin{displaymath}}
\newcommand{\edm}{\end{displaymath}}
\newcommand{\beq}{\begin{equation}}
\newcommand{\beqa}{\begin{eqnarray}}
\newcommand{\beqas}{\begin{eqnarray*}}
\newcommand{\eeq}{\end{equation}}
\newcommand{\eeqa}{\end{eqnarray}}
\newcommand{\eeqas}{\end{eqnarray*}}
\newcommand{\dd}{\textup{d}}

\newcommand{\E}{{\mathbb E}}
\newcommand{\R}{{\mathbb R}}
\newcommand{\C}{{\mathbb C}}

\newcommand{\real}{{\mathbb R}}
\newcommand{\SSS}{{\mathbb S}}
 
\newcommand{\uhat}{\Lambda_P^+ G^\C_\sigma}
\newcommand{\sym}{\mathcal{S}}

\newcommand{\ustar}{\Lambda_P^+ G^\C_\sigma}
\newcommand{\uc}{\Lambda G^\C_\sigma}

   \newtheorem{theorem}{Theorem}[section]

   \newtheorem{lemma}[theorem]{Lemma}
   \newtheorem{defn}[theorem]{Definition}

   \newtheorem{problem}{Problem}
   
 \theoremstyle{remark}
   \newtheorem{example}[theorem]{Example}

\numberwithin{equation}{section}

\begin{document}
\

\title[The Bj\"orling problem for CMC surfaces]{The Bj\"orling problem for non-minimal constant mean curvature surfaces}

\author{David Brander}
\address{Department of Mathematics\\ Matematiktorvet, Building 303 S\\
Technical University of Denmark\\
DK-2800 Kgs. Lyngby\\ Denmark}
\email{D.Brander@mat.dtu.dk}

\author{Josef F. Dorfmeister}
\address{TU M\"unchen\\ Zentrum Mathematik (M8), Boltzmannstr. 3\\
  85748, Garching\\ Germany}
\email{dorfm@ma.tum.de}

\begin{abstract}
The classical Bj\"orling problem  is to find the minimal
surface containing a given real analytic curve with tangent planes prescribed
along the curve.  We consider the generalization of this problem to non-minimal constant mean curvature
(CMC) surfaces, and show that it can be solved via the loop group formulation for
such surfaces.  The main result gives a way to compute the holomorphic potential
for the solution directly from the Bj\"orling data,
 using only elementary differentiation, integration and holomorphic extensions of
real analytic functions. Combined with an Iwasawa decomposition of the loop group,
this gives the solution, in analogue to Schwarz's formula for the minimal case.
Some preliminary examples of applications to the construction of CMC surfaces with special properties are given.
\end{abstract}




\maketitle


\section{Introduction}
In this article we consider the following:
\begin{problem} \label{problem1}
 Let $J \subset \real$ be an open interval. Let $f_0: J \to \E^3$ be a regular real analytic curve, with tangent vector field $f_0^\prime$. Let  $v:J \to \E^3$ be
 a non-vanishing analytic vector field along $f_0$
  such that the inner product $\left<v, f_0^\prime \right>= 0$ along the curve. Let $H$ be a non-zero real number.   Find all conformal
  constant mean curvature (CMC) $H$ immersions, $f: \Sigma \to E^3$,
 where $\Sigma$ is some open subset of $\C$ containing $J$,
 such that the restriction $f|_J$ coincides with $f_0$,
 and such that the tangent planes to the immersion 
 along $f_0$ are spanned by $v$ and $f_0^\prime$.
\end{problem}
This generalizes, to the case $H\neq0$, \emph{Bj\"orling's problem} for minimal surfaces,
proposed by E.G. Bj\"orling in 1844 and solved by H.A. Schwarz in 1890. See, for example, \cite{dhkw}. 

In the minimal case, there is a simple formula for the surface involving nothing
but integrals and analytic continuation of the initial data: specifically, if $f_0(x)$
is as above, and $n(x)$ is the unit normal to the prescribed family of tangent planes,
then, letting $\check n(z)$ and $\check f_0(z)$ be the analytic extensions of these
functions away from the curve, Schwarz's formula for the unique solution to the
Bj\"orling problem is 
\beq  \label{schwarz}
f(x,y) = \Re \left\{ \check f_0(z) - i \int_{x_0}^z \check n(w) \wedge \dd \check f_0(w) \right\}.
\eeq
The Schwarz formula has been an important tool in the study of minimal surfaces:
as a means to prove general facts about the surfaces,  such as the fact that if a minimal surface intersects
a plane perpendicularly, then this plane is a plane of symmetry of the surface.
It has been especially useful for constructing explicit examples of minimal surfaces;
for example, with certain symmetries (see \cite{dhkw}). For recent examples of applications to global problems of minimal surfaces
see \cite{galvezmira2004} and \cite{mira2006}.

The existence of the Schwarz formula is connected to the 
fact that the Gauss map of a minimal surface is holomorphic, and
to the Weierstrass representation for minimal surfaces.
Moreover, the Bj\"orling problem has also been studied in some other situations,
for example in works by G\'alvez, Mira and collaborators \cite{gm2005, gm2005-2, galvezetal2007},
where it is called the \emph{geometric Cauchy problem}.
 These geometric problems all have Weierstrass representations
in terms of holomorphic data, analogous to the case of minimal surfaces.

On the other hand, when $H\neq 0$ the Gauss map is merely harmonic, rather than holomorphic.

\subsection{Results of this article}
\subsubsection{Solution of the Bj\"orling problem}
In this article we make use of the so-called generalized 
Weierstrass representation for CMC surfaces of Dorfmeister, Pedit and Wu 
\cite{DorPW}, to show that the Bj\"orling problem can be solved 
for non-minimal CMC surfaces. The main result is Theorem \ref{mainthm}
which gives a method for computing the
 (unique) solution in terms of 
the given data $f_0$ and $v$, using  elementary integration and
 differentiation, 
analytic continuation of real functions, and an Iwasawa factorization
of the loop group. In fact our construction is highly analogous to the
Schwarz formula given above:  we take the loop group extended frame for the Gauss map determined
by the family of planes along the curve, extend this holomorphically away from
the curve, and then apply an Iwasawa decomposition of the complex loop group,
to obtain the ``real part" of the complexified frame.  This turns out to be an
 extended 
frame for the Gauss map of the desired surface, and the Sym-Bobenko formula retrieves
the CMC surface.  

A point that is perhaps not obvious in the procedure just described is the fact that the
loop group frame, which is defined in terms of the Hopf differential and the metric
of the surface (not the Gauss map alone), can actually be constructed given only the Bj\"orling data on a curve.

As a side-benefit,  Theorem \ref{mainthm}  also gives a new way to compute the
holomorphic data, a so-called \emph{holomorphic potential}, which
determine a given CMC surface, using nothing but analytic continuation
and integration. We call this new holomorphic potential a \emph{boundary potential},
because it is  an analytic extension of the Maurer-Cartan form of the (real) 
extended frame for the surface in question, which is computed
along a curve from the Bj\"orling data. 

 The construction of this boundary potential differs in an important respect
  from a previous method given by Wu 
\cite{wu1999} for finding the \emph{normalized} meromorphic potential 
for a CMC surface, because Wu's formula uses the holomorphic
part of the function $u$ - where the metric is given by $4e^{2u}(dx^2 + dy^2)$ -
 and this cannot be determined directly 
from the Bj\"orling data along a curve.

\subsubsection{Two-parameter families of CMC surfaces}
A new feature for the Bj\"orling problem  arises when one considers the 
non-minimal case: namely,  one now has the constant $H$ entering into the
construction, which can take any non-zero real value, and hence can be 
thought of as a parameter. Thus, in the 
non-minimal case, the solution to the Bj\"orling problem actually
gives a family of CMC surfaces through the given curve, varying continuously
with $H$.  In Example \ref{delaunayexample}, by taking the curve to be the
unit circle in the $x_1x_2$-plane and $v$ to point in the $x_3$ direction,
we obtain a family of potentials, representing a deformation of a sphere (minus two
points, strictly speaking) through
the Delaunay surfaces and the cylinder, all containing the same circle.

This leads to a corollary, Theorem \ref{famthm} which says that, given a CMC 1
surface $f$ and some choice of conformal parameterization, with metric
$4e^{2u}(\dd x^2 + \dd y^2)$, there is a 
continuous 1-parameter family
of CMC 1 surfaces $f^t$, where $t \in \R$, such that $f^1 = f$, 
and the Hopf
differential of $f^t$ is $Q_t \dd z^2= (2(1-t) e^{2\check u_0} + Q) \dd z^2$,
where $\check u_0$ is a holomorphic extension of 
$u|_J$, and $J$ is a given curve in the parameter domain. This  family  is different from
that  associated to the loop parameter, which scales the Hopf differential
by a complex constant.

\subsubsection{Applications to boundary value problems and to the construction of
CMC surfaces with symmetries}  
The experience of the use of Schwarz's formula in the study of minimal surfaces
indicates that Theorem \ref{mainthm} should be a useful tool for the non-minimal
case. We  consider some preliminary examples in  Section \ref{applications}.

Concerning boundary value problems, the problem of finding a CMC surface with a given boundary curve $\gamma$, can
always be formulated as a Bj\"orling problem, provided that $\gamma$ is a
real analytic curve.  This is because every CMC surface admits a conformal parameterization,
and, if the boundary is real analytic, the surface can be 
extended over the boundary \cite{muller2002}.   Now given such a curve $f_0$,
one can consider all possible vector fields $v$ for the Bj\"orling problem, and 
thus find expressions for the boundary potentials of all possible solutions.
In Section \ref{applications} we consider the simplest case of an open curve, 
and characterize all CMC surfaces which contain a straight line, the $x_1$-axis,
in terms of the angle made between the normal to the surface along the line
and the $x_3$ axis. Theorem \ref{linethm}
 gives the formula for the boundary potential for such a 
surface; we  then use the software CMCLab \cite{cmclab} to exhibit examples with
particular properties, such as the first two surfaces shown in Figure \ref{colorfig}.

\begin{figure}  
\begin{center}
\includegraphics[height=30mm]{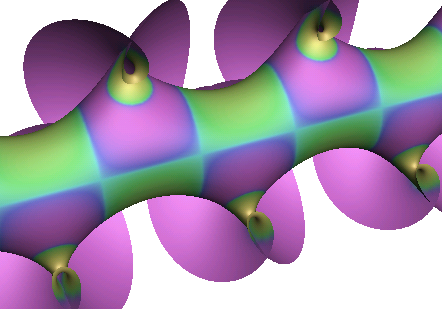} \hspace{2mm}
\includegraphics[height=30mm]{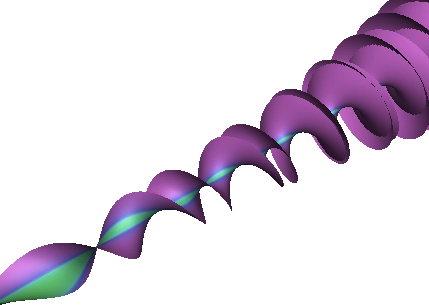} \hspace{2mm}
\includegraphics[height=30mm]{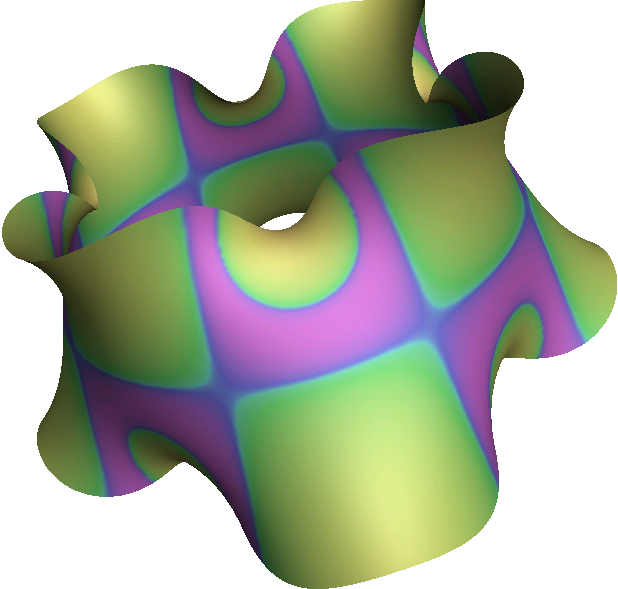}
\end{center}
\caption{Two conformally immersed non-minimal CMC surfaces which contain an entire straight line,
and one which contains a planar circle. The surfaces are  colored according to
 Gaussian curvature.}
\label{colorfig}
\end{figure}

The case that the boundary curve is closed, although  more complicated in general,
 can also be studied using this approach. In the final section,  we construct some examples of surfaces
 which contain a planar circle (Figures \ref{colorfig} right, \ref{circlefig1},
 \ref{circlefig2} and \ref{circlefig3}).



\section{The loop group formulation and DPW method}  \label{dpwsection}
In this section we summarize standard facts about CMC surfaces and their 
construction via integrable systems methods.
The loop group formulation for CMC surfaces in Euclidean space
${\mathbb E}^3$ evolved
from the work of Sym \cite{sym1985}, Pinkall 
and Sterling \cite{pinkallsterling},
and Bobenko \cite{bobenko, bobenko1994}.  
The Sym-Bobenko formula for CMC surfaces was given by 
Bobenko \cite{Bob:cmc, bobenko1994},
similar to the formula for constant negative Gauss surfaces of Sym \cite{sym1985}.  The DPW method for constructing all CMC surfaces from holomorphic data is due to Dorfmeister, Pedit and Wu \cite{DorPW}.  We give only an outline
of the DPW construction here, without reference to its more general purpose as a 
method for constructing pluriharmonic maps into symmetric spaces.

\subsection{The $SU(2)$-frame for a conformal immersion}\label{su2framesect}
It is known that every CMC surface admits a conformal parameterization.
Therefore, we first describe a standard $SU(2)$ frame for a conformally
parameterized surface.
For the Lie algebra $\mathfrak{su}(2)$, we work with the basis
\bdm
e_1 = \bbar 0 & -i \\ -i & 0 \ebar, \hspace{1cm}
e_2 = \bbar 0 & 1 \\ -1 & 0 \ebar, \hspace{1cm}
e_3 =  \bbar i & 0 \\ 0 & -i \ebar.
\edm
We identify  Euclidean 3-space $\E^3$ with $\mathfrak{su}(2)$, with inner
product given by $\langle X,Y \rangle = - \tfrac{1}{2} \text{trace} (XY)$,
giving the orthonormal relations
$ \langle e_i,e_j \rangle = \delta_{ij}$.
 
Let $\Sigma$ be a connected Riemann surface, and suppose $f: \Sigma \to \E^3$ is
a conformal immersion with (not necessarily constant) mean
curvature $H$. 
Conformality means we can choose  coordinates $z = x + iy$ and define
a function $u: \Sigma \to \real$ such that the metric is given by
\begin{equation}\label{eqn:dssquared}
\dd s^2 = 4e^{2u}(\dd x^2 + \dd y^2).
\end{equation}
A frame $F: \Sigma \to SU(2)$ is uniquely determined up to multiplication 
by $\pm 1$ by the conditions
\beq \label{framedef}
F e_1 F^{-1} = \frac{f_x}{|f_x|}, \hspace{1cm} 
F e_2 F^{-1} = \frac{f_y}{|f_y|}.
\eeq
The sign ambiguity is removed by choosing
 coordinates for   $\E^3$ so that
the frame $F$ satisfies $F(z_0) = I$, for some fixed point $z_0$.

A choice of unit normal vector is given by
$N = F e_3 F^{-1}$.  
The Hopf differential is defined to be $Q \dd z^2$, where
\bdm
Q:= \langle N,f_{zz} \rangle.
\edm 

The Maurer-Cartan form, $\alpha$, for the frame $F$ is defined by
\bdm
\alpha := F^{-1} \dd F = U \dd z + V \dd \bar{z}.
\edm

The mean curvature is 
 $H = \frac{1}{8}e^{-2u}\langle f_{xx} + f_{yy}, N \rangle$, 
and we have the formulae: $f_{zz} = 2u_z f_z + QN$, 
$f_{\bar{z} \bar{z}} = 2u_{\bar{z}}f_{\bar{z}} + \bar{Q}N$, 
$f_{z \bar{z}} = 2 H e^{2u}N$, and  
 \begin{equation}  \label{fzandfzbar}
f_z = -2i e^u F \cdot \bbar  0 & 1 \\ 0 & 0 \ebar
\cdot F^{-1}, \hspace{1cm}
f_{\bar z} = -2i e^u F \cdot \bbar  0 & 0 \\ 1 & 0 \ebar
\cdot F^{-1}. \end{equation} 
Differentiating these, one obtains  the following
\begin{lemma}  \label{withoutlambda}
The connection coefficients $U := F^{-1}F_z$ and 
$V := F^{-1}F_{\bar{z}}$ are given by 
\begin{equation}\label{UhatandVhat}
U = \frac{1}{2} \begin{pmatrix} u_z & -2 H e^u \\ 
                                   Q e^{-u}  & -u_z \end{pmatrix} 
, \hspace{1cm} V = \frac{1}{2} \begin{pmatrix} -u_{\bar z} & - \bar Q e^{-u}\\ 
                            2  H e^u & u_{\bar z} \end{pmatrix} 
. \end{equation}
The compatibility condition $\dd \alpha + \alpha \wedge \alpha = 0$
is equivalent to the pair of equations
\beqa  \label{compat}
\, && u_{z \bar z} + H^2 e^{2u} - \tfrac{1}{4}  |Q|^2 e^{-2u} = 0,
\label{compatibility1} \\
\, && Q_{\bar z} = 2 e^{2 u} H_z. \nonumber 
\eeqa
\end{lemma}

\subsection{CMC surfaces, the loop group and the Sym-Bobenko formula} \label{symformsection}
Now suppose we insert a parameter $\lambda$ into the $1$-form $\alpha$, defining 
the family
$\hat \alpha := \hat U  \dd z + \hat V \dd \bar{z}$, where
\begin{equation}\label{withlambda}
\hat U = \frac{1}{2} \begin{pmatrix} u_z & -2 H e^u \lambda^{-1} \\ 
                                   Q e^{-u} \lambda^{-1}  & -u_z \end{pmatrix} 
, \hspace{1cm} \hat V = \frac{1}{2} \begin{pmatrix} -u_{\bar z} & - \bar Q e^{-u} \lambda\\ 
                            2  H e^u \lambda & u_{\bar z}  \end{pmatrix} 
 . \end{equation}

The loop group characterization for CMC surfaces is contained in the
following fact, which is quickly verified by adding $\lambda$ at
the appropriate places in the compatibility conditions (\ref{compat}) above: 
\begin{lemma}
The $1$-form $\hat \alpha$ satisfies the Maurer-Cartan equation
\bdm
\dd \hat \alpha + \hat \alpha \wedge \hat \alpha = 0
\edm
 for all $\lambda \in \C \setminus \{ 0 \}$ 
if and only if the following two conditions both hold:
\begin{enumerate}
\item
$\left[ \, \dd \hat \alpha + \hat \alpha \wedge \hat \alpha \, \right]_{\lambda=1} = 0$,
\item
the mean curvature $H$ is constant.
\end{enumerate}
\end{lemma}
Note that it also follows from (\ref{compat}) that the Hopf differential
for a CMC surface is holomorphic.

Now 
 $\hat \alpha$ is a 1-form with values in the Lie algebra $Lie(\uu)$,
where $\uu$ is the loop group of maps from the unit circle $\SSS^1$  into
$G=SU(2)$, with a twisting condition that amounts to diagonal
and  off-diagonal components being respectively even and odd functions
in the $\SSS^1$-parameter $\lambda$.  The condition that the Maurer-Cartan
equation is satisfied for all $\lambda$ means that $\hat \alpha (\,\cdot \, ,\lambda)$
can be integrated for every $\lambda$ to obtain a map 
 $\hat F: \widetilde{\Sigma} \to \uu$ from the universal cover of
$\Sigma$ into $\uu$.

\begin{defn} \label{extendedframedef}
The map $\hat F: \widetilde {\Sigma} \to \uu$ obtained by integrating the 
above  1-form $\hat \alpha$, with the initial condition
 $\hat F(z_0) = I$,
 is called an \emph{extended frame}
for the CMC surface $f$.
\end{defn}

Note that $\hat F |_{\lambda=1}:  \widetilde {\Sigma} \to SU(2)$ coincides with the original frame $F$.

If $H$ is any nonzero real constant, the Sym-Bobenko formula, at $\lambda \in \SSS^1$,
 is given by:
\beq \label{symformula}
 \sym_{\lambda}(\hat F) :=  -\frac{1}{2H}  \left( \, \hat F e_3 \hat F^{-1} + 2 i \lambda 
\partial_\lambda \hat F \cdot \hat F^{-1} \, \right). 
\eeq
Note: Setting $\lambda = e^{it}$, we have $\frac{\partial}{\partial t} = 
i \lambda \frac{\partial}{\partial \lambda}$.  Hence  
$ \sym_{\lambda}(\hat F)$  takes 
values in  $\mathfrak{su}(2) = \E^3$.


\begin{theorem} \cite{Bob:cmc, bobenko1994} \label{symthm} $\,$ 
\begin{enumerate}
\item \label{symthm1} 
Given  a CMC $H$ surface, $f$, with  extended frame $\hat F: \widetilde \Sigma \to \uu$,
described above, the original surface $f$ is recovered by the formula
\beq  \label{normalizedsym}
f(x,y) =  \mathcal{S}_1 (\hat F(z))
  -\mathcal{S}_1(\hat F(z_0))
  +  f(z_0).
\eeq
For other values  
$\lambda_0 \in \SSS^1$,
 $\sym_{\lambda_0} (\hat F): \widetilde \Sigma \to \E^3$ is also a CMC $H$ surface 
in $\E^3$,
with 
Hopf differential given by $\lambda_0^{-2} Q$.

\item \label{symthm2}
Conversely, given a map $\hat F: \widetilde \Sigma \to \uu$, the Maurer-Cartan
of which has 
coefficients of the form given by (\ref{withlambda}), the map
 $\sym_{\lambda_0} (\hat F): \widetilde \Sigma \to \E^3$
obtained by the Sym-Bobenko formula is a CMC $H$ immersion into $\E^3$.
\item \label{symthm3}
If $D(z)$ is any diagonal matrix valued function, constant in $\lambda$, then
$\sym_\lambda(\hat F D) = \sym_\lambda(\hat F)$.
\end{enumerate}
\end{theorem}
\begin{proof}
For item (\ref{symthm1}), set $\hat f^\lambda : =  \sym_{\lambda}(\hat F)$
then compute that $\hat f^1_z=f_z$ and $\hat f^1_{\bar z}=f_{\bar z}$, 
so $f$ and $\hat f^1$ are the same surface up to translation. The formula
(\ref{normalizedsym}) follows immediately.
For other values of $\lambda$, see item (\ref{symthm2}). 
To prove (\ref{symthm2}), one computes $\hat{f}^{\lambda_0}_z$ and $\hat{f}^{\lambda_0}_{\bar{z}}$, 
and then the metric, the Hopf differential and the mean curvature.
Item (\ref{symthm3}) of the theorem is obvious.
\end{proof}


\subsection{The DPW construction}\label{dpwmethodsect}
Let $\uc$ denote the group of loops in $G^\C = SL(2,\C)$ with
the twisting described above. Let $\Lambda^{+}G_\sigma^\C$ denote
the subgroup of loops which extend holomorphically to the unit disc.
For the purpose of normalizations, we also use the subgroup
\bdm
 \ustar := \{ B \in \Lambda^+G^\C_\sigma ~|~ B(0) = \tiny{\bbar \rho & 0 \\ 0 & \rho^{-1} \ebar}, ~\rho \in \real, ~\rho>0 \}.
\edm

In order to describe the DPW method, we need the following standard
loop group decomposition, which allows one to write a $G^\C$-valued
loop as a product of a $G$-valued loop and a loop which extends
 holomorphically
to the unit disc.

\begin{theorem} \label{iwasawathm}
\cite{PreS, DorPW} (The Iwasawa Decomposition). Any element $g$ of $\uc$ can be uniquely expressed as a product
\bdm
g = FB, \hspace{2cm} F \in \uu, \hspace{1cm} B \in \ustar. 
\edm
The factors $F$ and $B$  depend real analytically on $g$.
\end{theorem}

We can now state the theorem of
 Dorfmeister, Pedit and Wu which is fundamental to what follows.
An explicit example is given below.

\begin{theorem}  \cite{DorPW} \cite{DorH:cyl} \label{dpwthm}
 (Generalized Weierstrass representation for 
CMC surfaces in $\E^3$).
Let 
\beqas
\hat \xi = \sum_{i=-1}^\infty A_i \lambda^i \dd z 
~~ \in ~ Lie(\uc) \otimes \Omega ^{1,0} (\Sigma), \\
A_{-1} = {\small{\bbar 0 & a_{-1} \\ b_{-1} & 0 \ebar}}, \hspace{1cm}
 a_{-1} \textup{ non-vanishing},
\eeqas
be a holomorphic 1-form over a simply-connected Riemann surface 
$\Sigma$. 

 Let 
$\hat \Phi :\Sigma \to \uc$ be 
a solution of 
\bdm
\hat \Phi^{-1} d\hat \Phi=\hat \xi.
\edm  
Consider the unique decomposition obtained from
applying Theorem \ref{iwasawathm}  pointwise on $\Sigma$: 
\beq \label{thmsplit}
\hat \Phi = \hat F \hat B, \hspace{1.5cm} \hat F: \Sigma \to \uu, \hspace{.5cm}
 \hat B : \Sigma \to  \ustar.
\eeq
  Then for any $\lambda_0 \in \SSS^1$, the map
 $\sym_{\lambda_0} (\hat F):  \Sigma \to \E^3$
 given by
the Sym-Bobenko formula \eqref{symformula}, is a conformal CMC $H$ immersion.

Conversely, let $\Sigma$ be a noncompact Riemann surface.  Then 
any nonminimal conformal CMC immersion 
from $\Sigma$ into $\E^3$ can be constructed in this manner, 
using a holomorphic potential $\hat \xi$ that is well-defined on $\Sigma$.   
\end{theorem}

To prove  the first direction in
 Theorem \ref{dpwthm}, one can show, using the fact that $\hat \xi = \lambda^{-1} A_{-1} + ...$ and that $\hat F$ is unitary, that we can 
 write
 \beq  \label{dpwmcf}
 \hat F^{-1} \dd \hat F = \bbar  c & a \lambda^{-1} \\ b \lambda^{-1} & -c \ebar \dd z  + \bbar  -\bar c & -\bar b \lambda \\ - \bar a \lambda & \bar c \ebar \dd \bar z.
\eeq
  Setting $f = \sym_{\lambda_0} {\hat F}$, one then
 computes that
 \bdm
 \hat F^{-1} f_z \hat F = \bbar 0 & -4ia\lambda_0^{-1}\\0&0\ebar,
 \hspace{1cm} \hat F^{-1} f_{\bar z} \hat F =
 \bbar 0 & 0 \\ -4i \bar a \lambda_0 & 0 \ebar.
\edm
 Post-multiplying the frame $\hat F$ by a diagonal matrix which is
independent of $\lambda$ (and therefore does not change $f$),
we can assume that $a \lambda_0^{-1}$ is real and positive, and write 
$a\lambda_0^{-1}= \bar a \lambda_0 = \frac{1}{2}e^{u}$.
Then, comparing $\hat F^{-1} f_z \hat F$ and $\hat F^{-1} f_{\bar z} \hat F$
with (\ref{fzandfzbar}), it  follows that $f$ is conformally immersed
and $\hat F$ is the coordinate frame defined in Section \ref{su2framesect}.
Moreover, the above expression (\ref{dpwmcf}) for the Maurer-Cartan 
form of $\hat F$ means that $\hat F$ satisfies the requirements of
Theorem \ref{symthm} to be the frame for a CMC surface.
 The condition $a_{-1} \neq 0$ is
the regularity condition, and it is also true that the surface has 
umbilics at points where $b_{-1}=0$.

The map $\hat \Phi$ above is called a \emph{holomorphic extended frame for $f$}.
Note that a holomorphic extended frame is by no means unique: however,
if we allow  \emph{meromorphic} frames, then there is a unique frame where
$\hat \xi$ is of the form $A_{-1} \lambda^{-1} \dd z$.

\begin{example}\label{exa:cylinders} A cylinder.  
Let 
\bdm
\hat \xi = {\small{\bbar  0 & \lambda^{-1} dz \\ \lambda^{-1} dz & 0\ebar}},
\edm
 on $\Sigma = \C$.  Then one solution $\hat \Phi$ of
 $d\hat \Phi = \hat \Phi \hat \xi$ is 
\bdm
\hat \Phi = 
\exp \Big\{ \bbar 0 &  z \lambda^{-1} \\ z \lambda^{-1} &0 \ebar \Big\}, 
\edm
which has the Iwasawa splitting
$\hat \Phi = \hat F \cdot \hat B$,
where
\bdm
\hat F = \exp \Big\{ 
\bbar  0 & z \lambda^{-1} - \bar z \lambda \\ z \lambda^{-1} - \bar z \lambda & 0\ebar \Big\}, \hspace{1cm} 
\hat B = \exp  \Big\{ \bbar 0 & \bar z\lambda \\ \bar z\lambda & 0 \ebar \Big\} ,
\edm
take values in $\uu$ and $\uhat$ respectively.  
The Sym-Bobenko formula $\sym_1 (\hat F)$ gives the surface 
\bdm
\frac{-1}{2H} \cdot 
[ 4 x, \, \sin(4 y), \,\cos(4 y) ],
\edm
in $\E^3 = \{ [x_1,x_2,x_0] := 
x_1e_1+x_2e_2+x_0e_3 \}$.  
\end{example}


\section{Solution of the Bj\"orling problem via the DPW method}  \label{cauchyprobsec}
We are now ready to consider Problem \ref{problem1}.
We are given a real analytic function $f_0: J \to \E^3$, which we want to extend to a conformally immersed CMC
surface $f :  \Sigma \to \E^3$, where $\Sigma$ is some open 
subset of $\C$ containing the set $J \times \{0\}$, which we also denote by $J$.  Such an extension is not unique, but we are also given the tangent plane to the surface along the image of $J$, in the form of a regular
 real analytic unit 
vector field $v: J \to \E^3$, such that $\langle v,  \frac{\dd f_0}{\dd x}\rangle = 0$. The 
required surface is to be tangent to the plane spanned by
 $\frac{\dd f_0}{\dd x}$ and $v$
along the curve.

  If  an extension exists, then we can choose an extended frame $\hat F: \Sigma \to \Lambda G_\sigma$, as described above,
such that $f$ is given by the Sym-Bobenko formula $\sym_1(\hat F)$.
We will construct $\hat F$ (and hence $f$) from the boundary data
given by $f_0$,  $v$.

The idea is  that it will be enough to find the Maurer-Cartan
form of $F$ in terms of the matrices $U$ and $V$ in (\ref{UhatandVhat})
\emph{only on the interval $J$}.  Then we can insert the parameter $\lambda$
as in (\ref{withlambda}), integrate this to find an expression, $\hat F_0$,
 for $\hat F$ along
$J$. Then, it turns out,  the holomorphic extension of this will be a
 holomorphic extended frame $\hat \Phi$ for the surface we seek.

 Examining the 
expression (\ref{UhatandVhat}) for $U$ and $V$, namely:
\bdm
U = F^{-1}F_z = \frac{1}{2} \begin{pmatrix} u_z & -2 H e^u \\ 
                                   Q e^{-u}  & -u_z \end{pmatrix} 
, \hspace{.5cm} V = F^{-1} F_{\bar z} =\frac{1}{2} \begin{pmatrix} -u_{\bar z} & - \bar Q e^{-u}\\ 
                            2  H e^u & u_{\bar z} \end{pmatrix} 
, \edm
we see that, in order to insert the parameter $\lambda$, and hence integrate 
to find the extended frame, it is necessary  and
sufficient to find 
the three functions:
\bdm
u,  \,\,\,\, \frac{\dd u}{\dd z}, \,\,\,\, \textup{and} \,\,\, \, Q
\edm
 along $J$, and so this is the
first goal.

The data $\frac{\dd f_0}{\dd x}$
and $v$ give us an $SU(2)$-frame along $J$ as
described in 
Section \ref{su2framesect}. Since we seek a conformal immersion, and $v$
is orthogonal to $\frac{\dd f_0}{\dd x}$,
we require that $\frac{\partial f}{\partial y} = 2e^u v$ along $J$,
and our standard frame from (\ref{framedef}) is determined 
(up to a sign), by
\bdm
F_0 e_1 F_0^{-1} = \frac{1}{2} {e^{-u}  \frac{\dd f_0}{\dd x}}, \hspace{1cm} 
F_0 e_2 F_0^{-1} = v,
\edm
where $u$ is yet to be determined. 
We can choose  coordinates 
for $\E^3$ such that $\frac{\dd f_0}{\dd x}(x_0)$
and $v(x_0)$ point in the directions of $e_1$ and $e_2$ respectively,
so that $F_0(x_0) = I$.

Now, by definition of $F_0$, it is necessary that
\bdm
\frac{\dd f_0}{\dd x} = 2 e^u F_0 \bbar 0 & -i \\ -i & 0 \ebar F_0^{-1},
\edm
and, taking the determinant, we obtain the  formula:

\beq \label{uexpression}
u =  \ln \Big(\frac{1}{2} \sqrt{\det (\frac{\dd f_0}{\dd x})} \Big),
\eeq
which can also be deduced by the requirement that
 $\langle \frac{\dd f_0}{\dd x}, \frac{\dd f_0}{\dd x} \rangle = 4e^{2u}$.
 
Now differentiating our frame $F_0$ along $J$
with respect to the parameter $x$, we can write
\beq  \label{Fhatmc}
F_0^{-1} (F_0)_x = \bbar  a & b \\ - \bar b & -a \ebar,
\eeq
where both $a$ and $b$ are known functions of $x$, and $a$ is pure imaginary.
According to Lemma \ref{withoutlambda},  the extension $F$ of
$F_0$ away from $J$ must satisfy:
\beqa
 F^{-1} F_x & =& U+V  \nonumber \\ 
     &=& \frac{1}{2} \begin{pmatrix} u_z - u_ {\bar z}& -2 H e^u -\bar Q e^{-u}\\ 
                    Q e^{-u} +2He^u & -u_z +u_{\bar z} \end{pmatrix}, \label{finvfx}
\eeqa
and this must agree along $J$ with the expression (\ref{Fhatmc}) corresponding to $F_0$. The $(1,1)$ components give
\bdm
u_z  - u_{\bar z} =  2a.
\edm
On the other hand, we  have, by definition,
\bdm
u_z  + u_{\bar z} =u_x .
\edm
Adding these equations gives
\beq  \label{uzexpression}
 u_z =  a + \frac{1}{2}u_x,
\eeq
in terms of known functions along $J$.

Now we use the $(1,2)$ components of the matrices above to get
\beq
\label{Qexpression}
 Q = -2e^u ( \bar b + H e^u).
\eeq

We can find the extended frame $\hat F_0$ along $J$ by inserting
these expressions for  $u$, $u_z$ and $Q$ into the restriction of the
1-form given by the equations (\ref{withlambda}) to the real line, namely
\beq \label{alphahatzero}
\hat  \alpha_0 = 
 \frac{1}{2} \left( \begin{pmatrix} 0 & -2 H e^u  \\ 
                                   Q e^{-u}  & 0 \end{pmatrix} \lambda^{-1}
+
  \begin{pmatrix} u_z - u_{\bar z} & 0 \\ 
                                   0  & -u_z +  u_{\bar z}  \end{pmatrix}
+
\begin{pmatrix} 0  & - \bar Q e^{-u}\\ 
                            2  H e^u  & 0  \end{pmatrix} \lambda \right) \dd x, 
\eeq
and then integrating this along $J$ by solving the equation 
$\hat F_0^{-1} \dd \hat  F_0 = \hat \alpha_0$
with the initial condition $\hat F_0(x_0) = I$.

We can now state the main result of this article:

\begin{theorem} \label{mainthm}
 Let $\hat F_0 : J \to \uu$,
be the extended frame along $J$ constructed above.
  Then
\begin{enumerate}
\item
There exists an open set $\Sigma \subset \C$ containing $J$, and
a holomorphic  map $\hat \Phi:  \Sigma \to \uc$ such that the restriction
 $\hat \Phi |_J$ of  $\hat \Phi$ to $J$ is equal to   $\hat F_0$.

\item
The Maurer-Cartan form of $\hat \Phi$ has a Fourier expansion in $\lambda$:
\bdm
\hat \Phi^{-1} \dd \hat \Phi = \sum_{i=-1}^1 A_i \lambda^i \dd z, 
\hspace{1cm} [A_{-1}]_{11} \neq 0.
\edm

\item 
The surface $f:\Sigma \to \E^3$ obtained from $\hat \Phi$ via 
the pointwise Iwasawa decomposition $\hat \Phi = \hat F \hat B$, with 
$\hat F \in \uu$ and $\hat B \in \ustar$, followed by
the Sym-Bobenko formula:
\beq  
f(x,y) =  \mathcal{S}_1 (\hat F(z))
  -\mathcal{S}_1(\hat F(z_0))
  +  f_0(x_0),
\eeq
is of constant mean curvature $H$, 
restricts to $f_0$ along $J$, and is tangent along $J$ to the plane spanned by
$\frac{\dd f_0}{\dd x}$ and $v$.

\item
The surface $f$ so constructed is the unique solution to 
Problem \ref{problem1}, in the following sense: if $\tilde f$ is any other solution, then,
for every point $x_0 \in J$, there exists a neighbourhood
 $\mathcal{N} = (x_0 -\epsilon, x_0+ \epsilon) \times (-\delta, \delta) \subset \C$ of $z_0 = (x_0, 0)$ such that $f |_{\mathcal{N}}= \tilde f |_{\mathcal{N}}$.

\end{enumerate}

\end{theorem}

\begin{proof}
Item (1) and (2):  $\hat F_0$ is obtained by solving the equation
$\hat F_0 ^{-1} \dd \hat F_0 = \hat \alpha_0 = (\hat U + \hat V)dx$, with the initial condition
$\hat F(x_0) = I$.   Now $\hat \alpha_0$ is of the form (\ref{alphahatzero}).
By construction, the components of the three coefficient matrices of
 $\hat \alpha_0$  are all real analytic
along $J$.  Hence there is some open set $\Sigma \subset \C$, containing $J$, to
which they simultaneously  extend holomorphically.  Since the component
$[( \hat \alpha_0)_{-1}]_{11} = - H e^u$ is non-vanishing on $J$, we can arrange,
by choosing $\Sigma$ sufficiently small,
that this also holds for the holomorphic extension.  Substituting these holomorphic
extensions for their counterparts,
 and $\dd z$ for $\dd x$, into the expression above for $\hat  \alpha_0$
gives a holomorphic extension $\hat \alpha$ of $\hat \alpha_0$. 
Since $\hat \alpha$ has trace zero and is twisted, this holomorphic 1-form
takes values in the Lie algebra $Lie (\uc)$.  We can choose 
$\Sigma$ to be contractible,  and then the equation 
$\hat \Phi ^{-1} \dd \hat \Phi$, $\hat \Phi (z_0) = I$
 can be solved uniquely to obtain the required map
$\hat \Phi : \Sigma \to \uc$.\\

Item (3): that the CMC surface $f: \Sigma \to \E^3$ exists is assured by
Theorem \ref{dpwthm}, since $\hat \Phi ^{-1} \dd \hat \Phi$ has the required
form.  Now the surface $f$ is obtained by taking the unique Iwasawa decomposition
\bdm
\hat \Phi = \hat F \hat B,
\edm
where $\hat F \in \uu$, and $B \in \ustar$, and applying the Sym-Bobenko formula to $\hat F$.
Since   $\hat \Phi |_J = \hat F_0$, which takes values in $\uu$, it follows
that the splitting along $J$ is just $\hat \Phi = \hat F_0 \cdot I$.
In other words,  $\hat F |_J = \hat F_0$. Hence 
 $\sym_1 (\hat F |_J) = \sym_1 (\hat F_0)$, and this is shown to be
equal to $f_0$ by a computation, as in the proof of Theorem \ref{symthm}.
The fact that $f$ is tangent along $J$ to the plane spanned by 
$\frac{\dd f_0}{\dd x}$ and $v$ is built into the construction of
the frame $F_0$ along $J$.\\

Item (4): for uniqueness, it is enough to show locally that any two
CMC $H$ surfaces which are equal and tangent along a part of a curve
are the same surface.  This can be done by a maximal principle, or
can also be seen from the construction of $\hat F_0$ given here.
About any point $x_0 \in J$, we have given a canonical means to construct
a unique extended frame  $\hat F_0$ along $J$, with  $\hat F_0 (x_0) = I$.
Now use the Birkhoff decomposition \cite{PreS} of $\uc$ to write
\bdm
\hat F_0 (x) = \hat F^-_0 \cdot \hat F^+_0,
\edm
where $\hat F^+_0 \in \Lambda^+ G^\C_\sigma$, and $\hat F^-_0$ is a loop which 
extends holomorphically to the exterior of the unit disc in the 
Riemann sphere and is normalized
so that $\hat F^-_0 (\lambda = \infty) = I$. This can be done pointwise
on an open subset of $J$ containing $x_0$ because $\hat F_0(x_0)$ is in the
big cell. Then
 $\hat F^-_0$ is
uniquely determined by $\hat F_0$, depends real analytically on $x$ (see
\cite{DorPW}), and, it is straightforward to verify,
has a Maurer-Cartan form of a very simple form:
\bdm
(\hat F^-_0) ^{-1} \dd \hat F^-_0 = \bbar 0 & a_0 \\ b_0 & 0 \ebar \lambda^{-1} \dd x,
\hspace{1cm} a_0 \neq 0.
\edm
The real analytic functions $a_0$ and $b_0$ have unique holomorphic
 extensions $a$ and $b$ to some neighbourhood 
of $(x_0, 0)$ in $\C$, and putting these into the 1-form
\beq \label{normalized}
\hat \xi =  \bbar 0 & a \\ b & 0 \ebar \lambda^{-1} \dd z,
\eeq
 we see that
we have a potential for a CMC surface, as in Theorem \ref{dpwthm}.

On the other hand, if we are given two surfaces which solve the Bj\"orling
problem, we could just as well have constructed the extended frame for 
each of the two surfaces on some neighbourhood of the point
$z_0 = (x_0, 0)$ in $\C$, rather than restricting to the real line.
For each surface we  obtain a \emph{unique} map  $\tilde F^-$,
with $\tilde F^-(\lambda=\infty) = I$, and a Maurer-Cartan form
of the form (\ref{normalized}).  One can verify that this so-called 
\emph{normalized potential} is holomorphic, and, since the corresponding
holomorphic functions $\tilde a$ and $\tilde b$ agree, by construction,
 with $a_0$ and $b_0$ along $J$, it follows that they agree everywhere
and  the surface constructed from $\hat \xi$ is the original surface.
Hence the two original surfaces are the same.

\end{proof}

\begin{defn} The holomorphic extension $\hat \alpha$ of $\hat \alpha_0$
defined in the proof above will be called the \emph{boundary potential} for
the CMC surface in question.
\end{defn}


\subsection{Example}  \label{delaunayexample}
As a test case, 
we compute the solution when the the initial curve is a circle in a plane,
and the tangent plane along the curve is orthogonal to this plane.

We take the circle
\bdm
 f_ 0(x) = [\sin 2x, 0 , -\cos 2x] = \bbar -i \cos 2x & -i \sin 2x \\
               -i \sin 2x & i \cos 2x \ebar,
\edm
and the vector field
\bdm
v(y)= [0,1,0] = \bbar 0 & 1 \\ -1 & 0 \ebar.
\edm

Then $\frac{\dd f_0}{\dd x} (x) = 2i \bbar \sin 2x & -\cos 2x \\ -\cos 2x & -\sin 2x \ebar$, and using the expression (\ref{uexpression})
 we  must have
 \bdm
 u =  \ln \Big(\frac{1}{2} \sqrt{\det (\frac{\dd f_0}{\dd x})} \Big) = 0.
 \edm
 
To find $a$ and $b$ in $F_0^{-1} (F_0)_x$, 
 we need to convert the vector fields $\frac{\dd f_0}{\dd x} $ and $v$ into an 
$SU(2)$-frame  $F_0(x)$. The vector fields in question  are orthogonal, so, we
look for a conformal immersion with coordinates $z=x + iy$ and such that
$f_y$ is in the direction of $v$.  The frame according to the recipe is determined
 by 
\beq  \label{framesolve}
F_0 e_1 F_0^{-1} = \frac{1}{2}e^{-u} \frac{\dd f_0}{\dd x}, \hspace{1cm} 
F_0 e_2 F_0^{-1} =   v.
\eeq
Setting $F_0 = \small{\bbar A & B \\ - \bar B & \bar A \ebar}$, the first of
these two equations gives
\bdm
2 \Re (A \bar B) = -\sin 2x, \hspace{2cm} A^2-B^2 = \cos 2x,
\edm
and the second equation gives
\bdm
 \Im (A \bar B) = 0, 
 \hspace{2cm} A^2 + B^2 = 1.
\edm
The unique solution that satisfies $F_0(0) = I$ is the $SU(2)$-frame
\bdm
F_0 (x) = \bbar \cos x & -\sin x \\ \sin x & \cos x \ebar.
\edm

Now we equate
\bdm
F_0^{-1}  (F_0)_x = \bbar 0 & -1 \\ 1 & 0 \ebar = \bbar a & b \\ - \bar b & -a \ebar,
\edm
so that $a=0$ and $b=-1$.

Substituting into equations (\ref{uzexpression}) and (\ref{Qexpression}) we have,
along the real axis,
\bdm
u_z = 0, \hspace{1cm} Q = 2(1-H),
\edm
and the extended frame $\hat F_0$ along the $x$-axis is computed by integrating
the Maurer-Cartan form
\bdm
\hat \alpha_0 = \frac{1}{2} \left( \bbar 0 & -2H \\ 2(1-H) & 0 \ebar \lambda^{-1} 
    +  \bbar 0 & - 2(1-H) \\ 2H & 0 \ebar \lambda \right) \dd x.
\edm
Hence the boundary potential for the surface given by Theorem \ref{mainthm} is
\bdm
\hat \alpha (z) = \bbar 0 &  ~(H-1) \lambda - H \lambda^{-1} \\
     H \lambda - (H-1) \lambda^{-1} ~ & 0 \ebar \dd z.
\edm

When $H \neq 1$, this holomorphic potential satisfies 
the conditions to be that of a Delaunay surface 
(see \cite{kilian}, where the fact that the rotation parameter is $iy$ rather than
$x$ introduces a minus sign in the lower left corner).  One obtains a sphere when $H=1$, a cylinder when
 $H=\frac{1}{2}$, and unduloids and nodoids for other values of $H$.

\subsection{Two parameter families of CMC surfaces}
Observe that in the previous example, if we simply drop the $\frac{1}{H}$ term
in front of the Sym-Bobenko formula, we actually obtain a one-parameter family 
of CMC 1 surfaces, which deforms a sphere (minus two points) of radius $1$ 
continuously through a family of Delaunay surfaces to a cylinder of radius $\frac{1}{2}$.
Thus we see that if we are given a sphere as our initial object, and the circle in
the sphere, then we obtain a 1-parameter family of CMC-1 surfaces, by going
through the Bj\"orling construction starting with this circle, the tangent plane
to the sphere, inserting $H$ (now thought of as just a real parameter) into the
expression  (\ref{Qexpression}) for $Q$, constructing the boundary potential, 
and finally using the Sym-Bobenko formula without the $\frac{1}{H}$ factor.
Clearly we can do this for any surface and any given curve in the surface, so we have
the following corollary of Theorem \ref{mainthm}:
\begin{theorem} \label{famthm}
Let $\Sigma \subset \C$ be a simply connected domain, containing the
origin, and with coordinates $z=x+iy$. Suppose given a conformal  CMC 1 immersion $f: \Sigma \to E^3$, 
with metric given by $4e^{2u}(dx^2 + dy^2)$. Let $\Sigma^\prime \subset \Sigma$ be  any
simply connected open subset to which the functions $u_0(x) := u(x,0)$
and $\eta_0(x) := -i u_y(x,0) = (u_z - u_{\bar z})(x,0)$ extend
holomorphically, and denote the holomorphic extensions by $\check u_0$
and $\check \eta_0$.
 Then there exists a continuous
1-parameter family of conformal CMC 1 immersions $f^t: \Sigma^\prime\to \E^3$, with
Hopf differential given by
\bdm
Q_t = 2(1-t)e^{2\check u_0} + Q,
\edm
and such that $f^1 = f|_{\Sigma^\prime}$. The surfaces are related on the
real coordinate axis by the relation
\bdm
f^t(x,0) = t f^1(x,0).
\edm
The boundary potential for $f^t$ is
\bdm
\hat  \alpha^t = 
   \begin{pmatrix} \frac{1}{2} \check \eta _0 & - t e^{\check u_0} \lambda^{-1} - ((1-t)e^{\check u_0} + \frac{1}{2} \bar Q e^{-\check u_0}) \lambda \\ 
                                   ((1-t)e^{\check u_0} + \frac{1}{2}Q e^{-\check u_0})\lambda^{-1} +  t e^{\check u_0} \lambda  & - \frac{1}{2}\check \eta_0 \end{pmatrix}   \dd z. 
\edm
\end{theorem}
\begin{proof}
The proof is a matter of going through the construction above for the solution of 
the Bj\"orling problem, starting with  $f_0 = f(x,0)$. We have, for
equation (\ref{finvfx}) 
\bdm
 F^{-1} F_x 
     = \frac{1}{2} \begin{pmatrix} u_z - u_ {\bar z}& -2  e^u -\bar Q e^{-u}\\ 
                    Q e^{-u} +2e^u & -u_z +u_{\bar z} \end{pmatrix},
 \edm
 where $Q$ and $u$ are those of the given surface $f$ along the real axis,
 so that $b=-  e^u -\frac{1}{2}\bar Q e^{-u}$, and substituting this into the 
 expression (\ref{Qexpression}), and the parameter $t$ instead of $H$, we
 obtain the above expression for $Q_t$.  Finally, we construct the surface
 from Theorem \ref{mainthm}, and scale the result by a factor $t$, so that
 our surface is CMC 1, rather than CMC $t$.
 \end{proof}
 
Note that this can be done for any open curve in the coordinate
domain of a surface, by changing conformal coordinates so that this curve becomes a part of  the $x$-axis.


\section{Applications to boundary value problems and surfaces with symmetries}  \label{applications}
By a result of F M\"uller (Theorem 5 of \cite{muller2002}),  given a conformally parameterized 
CMC surface with boundary, which is continuous at the boundary, and where the boundary curve in $\E^3$ is an embedded real analytic curve, the 
surface can be extended analytically across the boundary.  Therefore, for such boundary curves, we may always assume that the boundary is contained in the surface,
and have the possibility of treating it as a Bj\"orling problem, by considering all possible 
tangent planes along this curve.

\subsection{CMC surfaces which contain a line segment}
In this section we use the boundary potential to describe all simply connected
CMC surfaces which contain
a given line or line segment.
\begin{theorem} \label{linethm}
Let $J= (\alpha, \beta)$ be an open interval, where we allow $\alpha = -\infty$ and $\beta= \infty$.
Let $l= (2 \alpha, 2 \beta) \times \{0\} \times \{0\} \subset \E^3$.
\begin{enumerate}
\item
Given a real analytic function $\theta_0 : J  \to \real$, with
$\theta_0(x_0) = 0$ for some $x_0 \in J$,
let $\Sigma$ be any simply connected domain in $\C$, containing $J \times \{0\}$, to 
which $\frac{\dd \theta_0}{\dd x}$ extends analytically. Denote this
analytic extension by $\check \theta_x$.
 Then, for any real $H \neq 0$, there is a conformally parameterized CMC $H$ surface $f :\Sigma \to E^3$ which 
maps $(\alpha,\beta) \times \{0\} \subset \C$ to $l$, via the map $(x,0) \mapsto [2x,0,0]$. The Hopf differential of this surface is given by $Q \dd z$, where 
\bdm
Q = - i \check \theta_x - 2 H.
\edm
The boundary potential of the surface is given by
\bdm
\hat \alpha =  \bbar 0 & -H \lambda^{-1} + (- \frac{1}{2} i \check \theta_x +H) \lambda \\
(-\frac{1}{2} i \check \theta_x -H)\lambda^{-1} + H \lambda \ebar \dd z.
\edm
\item Conversely, any simply connected non-minimal CMC surface in $\E^{3}$ which contains
the segment $l$, can be represented, on an open subset containing  $l$,
 this way.  For each $x \in J$, the value
$\theta_0(x) \mod 2\pi$ is the
angle between the normal to the surface at the point $[2x,0,0]$ and some fixed
line in the plane spanned by $e_2$ and $e_3$.
\end{enumerate}
\end{theorem}
\begin{proof}
Item 1:
This amounts to interpreting the function $\theta_0$ as 
the vector field $v$ for the Bj\"orling problem for the map
 $f_0: J \to \E^3$,
\bdm
f_0(x) = [2x,0,0] = \bbar 0 & -2ix \\ -2i x & 0 \ebar.
\edm
Since $f_0$ is always tangent to the $x_1$-axis, the map 
 $v: J \to E^3$ determined by $\theta_0$ via: 
\bdm 
v(x) = [0, \cos \theta_0, \sin \theta_0] = \bbar i \sin \theta_0 & \cos \theta_0 \\
         -\cos \theta_0  & - i \sin \theta_0 \ebar,
\edm
is orthogonal to $\frac{\dd f_0}{\dd x}$ for all $x$.
The normalization $\theta_0(x_0) = 0$ is equivalent to choosing 
$v(x_0) = e_2 = \bbar 0 & 1 \\ -1 & 0\ebar$, which will give us our standard 
normalization of the frame: $F_0(x_0) = I$.

According to equation $\ref{uexpression}$, we will have
$u = \ln (\frac{1}{2} \sqrt{\det \frac{\dd f_0}{\dd x}}) = 0$,
and solving the equations (\ref{framesolve}), we obtain the unique $SU(2)$ frame
mapping $x_0$ to the identity:
\bdm
F_0 = \bbar \cos \frac{\theta_0}{2} & -i \sin \frac{\theta_0}{2} \\ -i \sin \frac{\theta_0}{2} & \cos \frac{\theta_0}{2} \ebar.
\edm
Hence
\bdm
F_0^{-1} (F_0)_x = -\frac{i}{2} \frac{\dd \theta_0}{\dd x} \bbar 0 & 1  \\ 1 & 0 \ebar,
\edm 
and  the formulae (\ref{uexpression}), (\ref{uzexpression}) and (\ref{Qexpression})
are
\bdm
u=u_z=0, \hspace {1cm} Q = -i \frac{\dd \theta_0}{\dd x} - 2 H,
\edm
which, extending holomorphically, gives the expression for the Hopf differential above.  Finally, substituting these into the expression (\ref{alphahatzero}) for 
$\hat \alpha_0$, and extending holomorphically, we obtain the above 
expression for the boundary potential $\hat \alpha$.

Item 2: For the converse, on an open set containing $l$, one can always 
choose conformal coordinates $(x,y)$ such that $f$  maps 
$J \times \{0\} \to l$ by the function $f((x,0)) = [2x,0,0]$. 
Fix a point $[2x_0,0,0] \in l$ and change coordinates
of $\E^3$ so that $\frac{\partial f}{\partial y}(x_0,0)$ is in the $e_2$ direction.
Then the frame $F_0$ will be given as above, where $\theta_0$ is the angle
between the normal direction and our fixed choice of 
$e_3$ direction. By Theorem \ref{mainthm},
 a non-minimal CMC
surface is determined by its Bj\"orling data, so the Hopf differential and 
boundary potential  stated above, are those of the original surface.  

\end{proof}


\begin{figure}   
\begin{center}
\includegraphics[height=28mm]{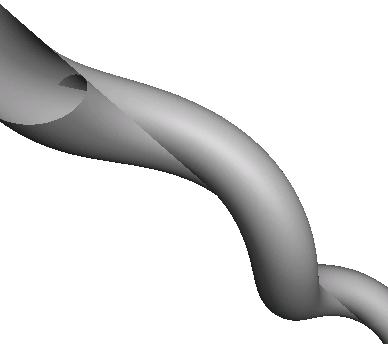} \hspace{4mm}
\includegraphics[height=28mm]{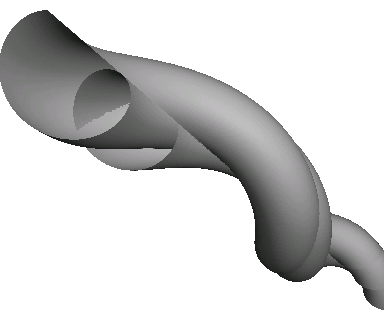} \hspace{4mm}
\includegraphics[height=28mm]{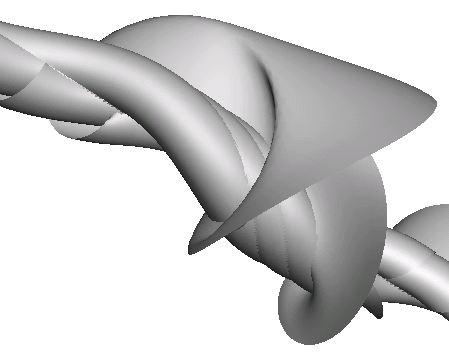}
\end{center}
\caption{Three partial plots of one  surface, containing a line $l$.
The surface normal along the line $l$ 
rotates around $l$ at constant speed with respect to the arc-length 
parameter of $l$.}
\label{twistyfig}
\end{figure}


\subsection{Examples containing a straight line}
If we choose $\theta_0$ to be constant, we of course get a cylinder, with
boundary potential:
\bdm
\hat \alpha = 
  \bbar 0 & -H \lambda^{-1}  +H \lambda \\
-H\lambda^{-1} + H \lambda & 0 \ebar \dd z.
\edm
To check that this potential really does give a cylinder, observe that 
the holomorphic extended frame obtained by integrating $\hat \alpha$ is 
$\hat \Phi(z) = \exp((-H \lambda^{-1}  +H \lambda) z A)$, where $A= \tiny \bbar 0 & 1 \\ 1 & 0 \ebar$.
Since this can be written as $\hat \Phi(z) = \exp(-H \lambda^{-1} z A) \cdot
\exp(H \lambda z A)$ and the second matrix is in $\uhat$, the second factor has no effect on the term
$F$ in the 
Iwasawa decomposition $\hat \Phi = F B$, $F \in \uu$, $B \in \uhat$. Hence the surface  obtained
from this potential is the same as the one obtained from the potential $\xi = -H \lambda^{-1} A \dd z$,
which was shown to be a cylinder in Example \ref{exa:cylinders}.

If we choose $\theta_0 = 2x$, to obtain a surface the normal to which rotates about
the line in a spiral (Figure \ref{twistyfig}), the boundary potential is 
\bdm
\hat \alpha =  \bbar 0 & -H \lambda^{-1} + (- i  +H) \lambda \\
(- i-H)\lambda^{-1} + H \lambda & 0 \ebar \dd z.
\edm


\begin{figure}  
\begin{center}
\includegraphics[height=35mm]{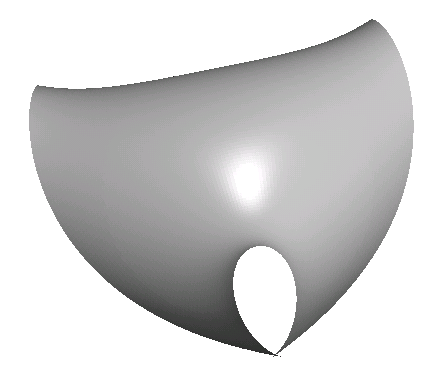} \hspace{4mm}
\includegraphics[height=35mm]{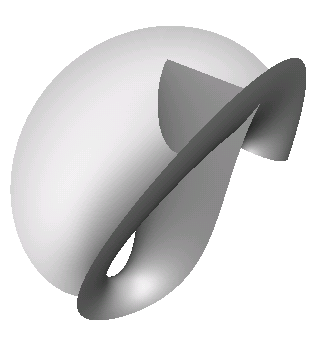}  \hspace{4mm}
\includegraphics[height=35mm]{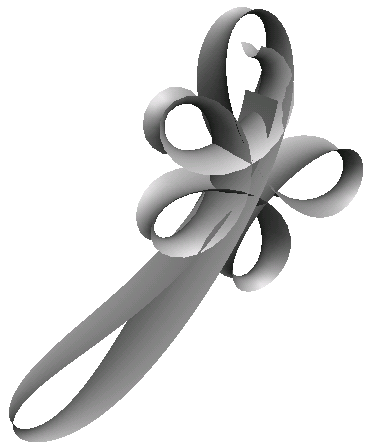}\\
\end{center}
\caption{Three partial plots of a  surface, containing a line $l$, 
the normal to which rotates with constantly increasing angular velocity around $l$.
It is conformally parameterized by an immersion $f: \C \to \E^3$, which maps the real line
to $l$. It has  exactly one umbilic at $z=i$, around the spot of white light
on the first image.  The last two  images show the image of a narrow strip around the positive imaginary-axis. The image of a narrow strip around the 
positive real axis is shown in Figure \ref{colorfig} (center).}
\label{acceleratingtwistfig}
\end{figure}

If we choose $\theta_0 = x^2$, to obtain a surface the normal to which rotates 
with constantly increasing angular velocity about
the line (Figure \ref{acceleratingtwistfig}), the boundary potential is 
\bdm
\hat \alpha =  \bbar 0 & -H \lambda^{-1} + (- i z  +H) \lambda \\
(- i z -H)\lambda^{-1} + H \lambda & 0 \ebar \dd z.
\edm





If we choose $\theta_0 = \frac{\pi}{8} \sin^2(x)$, then we  obtain a surface 
the normal to which maintains a small and periodic angle with the $x_3$ direction,
along the line $l$ (Figure \ref{colorfig}, first image). The boundary potential is 
\bdm
\hat \alpha =  \bbar 0 & -H \lambda^{-1} + (- i \frac{\pi}{8}\cos z \sin z  +H) \lambda \\
(- i  \frac{\pi}{8}\cos z \sin z -H)\lambda^{-1} + H \lambda & 0\ebar \dd z.
\edm

\subsection{Examples of CMC surfaces containing a planar circle}
Using an analogous approach for the circle, we can easily 
construct CMC surfaces containing a circle in a plane.
The surfaces shown in Figures \ref{circlefig1} - \ref{circlefig3} 
all come from boundary potentials of the form
\bdm
\alpha = \frac{1}{2} \bbar \frac{1}{z}(\cos (2 \check \theta) -1) &
  -H \lambda^{-1} + (\sin 2 \check \theta + H - 2 i \check \theta^\prime) \lambda \\
  \frac{1}{z^2}(\sin 2 \check \theta + H + 2 i \check \theta^\prime)\lambda^{-1}
    - \frac{1}{z^2} H \lambda &
     - \frac{1}{z}(\cos (2 \check \theta) -1) \ebar \dd z,
\edm
where $\theta: \real \to \real$ satisfies $\theta(t+2 \pi) = \theta(t) + 2k \pi$
for some integer $k$, and $\check \theta$ is the analytic 
extension of $\theta(-i \ln z)$ to an annulus around $S^1$.  
For Example 1, we used $\theta(t)$, up to a translation,
 proportional to $\sin(t)$; for the other examples
$\theta(t)$ is, also up to a translation, proportional to $\sin^2(k t)$ for some integer $k$.

\begin{figure}
\begin{center}
\includegraphics[height=32mm]{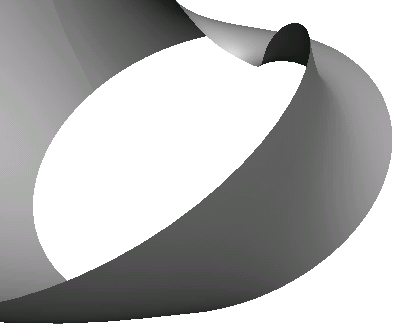} \hspace{1.5cm} 
\includegraphics[height=34mm]{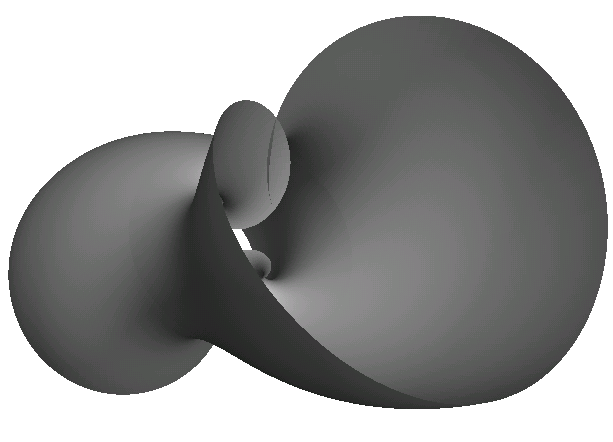}
\end{center}
\caption{This surface contains a planar circle and has a reflective symmetry about the origin.}
\label{circlefig1}

\end{figure}

\begin{figure}
\begin{center}
\includegraphics[height=31mm]{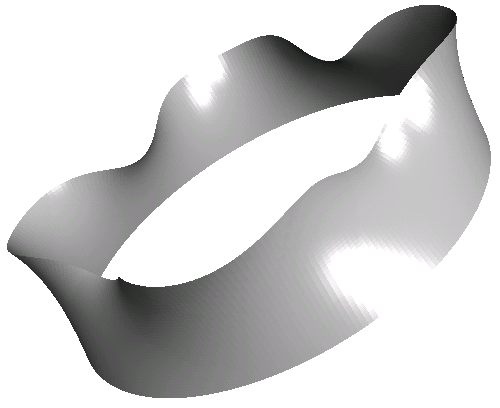}  \hspace{4mm}
\includegraphics[height=32mm]{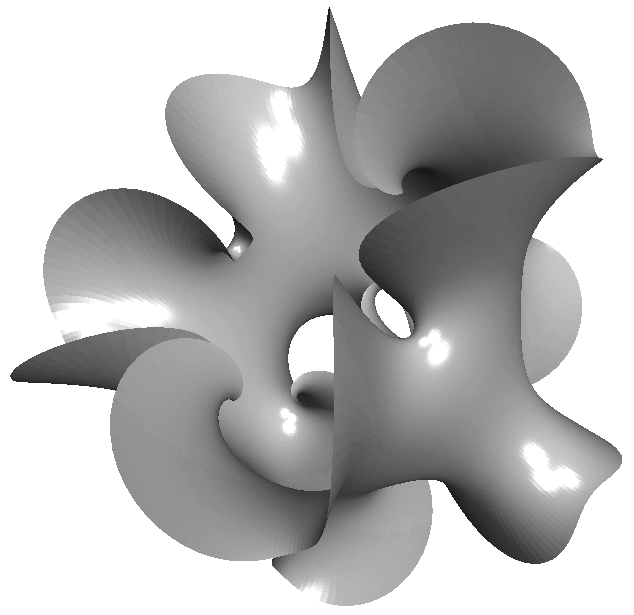}  \hspace{4mm}
\includegraphics[height=32mm]{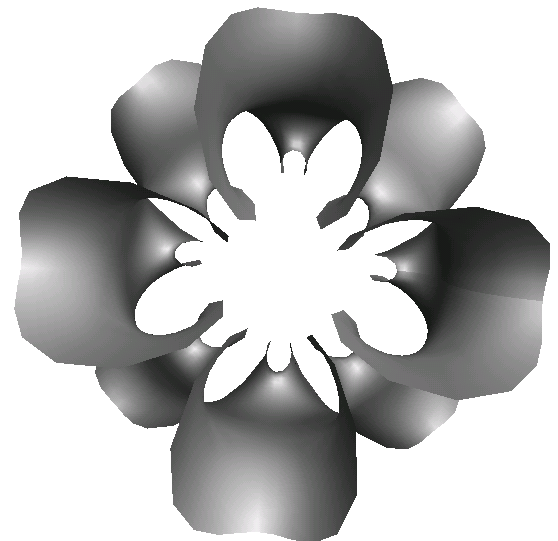} 
\end{center}
\caption{A CMC surface invariant under rotations of $\pi/2$ around the $x_3$-axis.} \label{circlefig2}
\end{figure}

\begin{figure}
\begin{center}
\includegraphics[height=38mm]{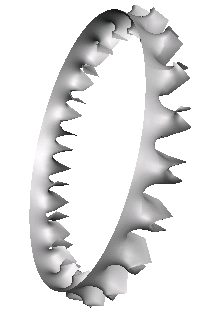} \hspace{25mm} 
\includegraphics[height=38mm]{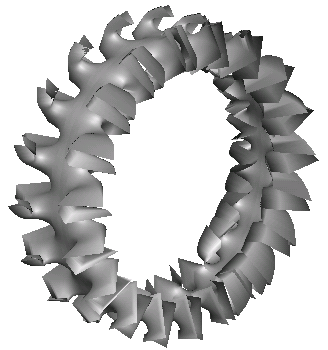}
\end{center}
\caption{A CMC surface containing a planar circle and  which is 
invariant under rotations of $\pi/10$ around the $x_3$-axis.}
\label{circlefig3}
\end{figure}

\pagebreak

\providecommand{\bysame}{\leavevmode\hbox to3em{\hrulefill}\thinspace}
\providecommand{\MR}{\relax\ifhmode\unskip\space\fi MR }
\providecommand{\MRhref}[2]{%
  \href{http://www.ams.org/mathscinet-getitem?mr=#1}{#2}
}
\providecommand{\href}[2]{#2}


\begin{thebibliography}{10}

\bibitem{galvezetal2007}
J~Aledo, R~Chaves, and J~G\'alvez, \emph{The {C}auchy problem for improper
  affine spheres and the {H}essian one equation}, Trans. Amer. Math. Soc.
  \textbf{359} (2007), 4183--4208.

\bibitem{bobenko}
A~I Bobenko, \emph{All constant mean curvature tori in ${R}^3$, ${S}^3$,
  ${H}^3$ in terms of theta-functions}, Math. Ann. \textbf{290} (1991),
  209--245.

\bibitem{Bob:cmc}
\bysame, \emph{Surfaces of constant mean curvature and integrable equations},
  Uspekhi Mat. Nauk \textbf{46:4} (1991), 3--42, Translation in: Russian Math.
  Surveys, \textbf{46} (1991), 1-45.

\bibitem{bobenko1994}
\bysame, \emph{Surfaces in terms of 2 by 2 matrices. {O}ld and new integrable
  cases}, Harmonic maps and integrable systems, Aspects Math., no. E23, Vieweg,
  1994, pp.~83--127.

\bibitem{dhkw}
U~Dierkes, S~Hildebrandt, A~K\"uster, and O~Wohlrab, \emph{Minimal surfaces.
  {I}. {B}oundary value problems}, Grundlehren der {M}athematischen
  {W}issenschaften, vol. 295, Springer-Verlag, 1992.

\bibitem{DorH:cyl}
J~Dorfmeister and G~Haak, \emph{Construction of non-simply connected {CMC}
  surfaces via dressing}, J. Math. Soc. Japan \textbf{55} (2003), no.~2,
  335--364.

\bibitem{DorPW}
J~Dorfmeister, F~Pedit, and H~Wu, \emph{Weierstrass type representation of
  harmonic maps into symmetric spaces}, Comm. Anal. Geom. \textbf{6} (1998),
  633--668.

\bibitem{galvezmira2004}
J~G\'alvez and P~Mira, \emph{Dense solutions to the {C}auchy problem for
  minimal surfaces}, Bull. Braz. Math. Soc. (N.S.) \textbf{35} (2004),
  387--394.

\bibitem{gm2005}
\bysame, \emph{The {C}auchy problem for the {L}iouville equation and {B}ryant
  surfaces}, Adv. Math. \textbf{195} (2005), 456--490.

\bibitem{gm2005-2}
\bysame, \emph{Embedded isolated singularities of flat surfaces in hyperbolic
  3-space}, Calc. Var. Partial Differential Equations \textbf{24} (2005),
  239--260.

\bibitem{kilian}
M~Kilian, \emph{On the associated family of {D}elaunay surfaces}, Proc. Amer.
  Math. Soc. \textbf{132} (2004), 3075--3082.

\bibitem{mira2006}
P~Mira, \emph{Complete minimal {M}\"obius strips in {${\mathbb R}^n$} and the
  {B}j\"orling problem}, J. Geom. Phys. \textbf{56} (2006), 1506--1515.

\bibitem{muller2002}
F~M\"uller, \emph{Analyticity of solutions for semilinear elliptic systems of
  second order}, Calc. Var. Partial Differential Equations \textbf{15} (2002),
  257--288.

\bibitem{pinkallsterling}
U~Pinkall and I~Sterling, \emph{On the classification of constant mean
  curvature tori}, Ann. of Math. (2) \textbf{130} (1989), 407--451.

\bibitem{PreS}
A~Pressley and G~Segal, \emph{Loop groups}, Oxford Mathematical Monographs,
  Clarendon Press, Oxford, 1986.

\bibitem{cmclab}
N~Schmitt, \emph{{CMCL}ab}, http://www.gang.umass.edu/software.

\bibitem{sym1985}
A~Sym, \emph{Soliton surfaces and their applications}, Geometric aspects of the
  {E}instein equations and integrable systems, Lecture Notes in Physics, vol.
  239, Springer, 1985, pp.~154--231.

\bibitem{wu1999}
H~Wu, \emph{A simple way for determining the normalized potentials for harmonic
  maps}, Ann. Global Anal. Geom. \textbf{17} (1999), 189--199.

\end{thebibliography}
\end{document}